\documentclass[12pt]{amsart}
\usepackage{amscd}
\usepackage{amsmath}
\usepackage{amssymb}
\usepackage{amsthm}
\usepackage{url}
\usepackage{comment}
\newtheorem{theorem}{Theorem}[section]
\newtheorem{lemma}[theorem]{Lemma}
\newtheorem{proposition}[theorem]{Proposition}
\newtheorem{corollary}[theorem]{Corollary}

\newtheorem{conjecture}[theorem]{Conjecture}

\theoremstyle{definition}

\newtheorem*{acknowledgement}{Acknowledgment}
\theoremstyle{remark}

\numberwithin{equation}{section}

\def\Fc{{\mathcal F}}
\def\Pc{{\mathcal P}}

\def\NZQ{\mathbb}               

\def\ZZ{{\NZQ Z}}
\def\RR{{\NZQ R}}

\begin{document}
	\title{Best possible lower bounds on the coefficients of Ehrhart polynomials}
	\author{Akiyoshi Tsuchiya}
	\subjclass[2010]{52B20}
	\thanks{
		{\bf Keywords:}
		Ehrhart polynomial, Integral convex polytope, $\delta$-vector.
	}
	\address{Akiyoshi Tsuchiya,
		Department of Pure and Applied Mathematics,
		Graduate School of Information Science and Technology,
		Osaka University,
		Toyonaka, Osaka 560-0043, Japan}
	\email{a-tsuchiya@cr.math.sci.osaka-u.ac.jp}
	\begin{abstract}
		For an integral convex polytope $\Pc \subset \RR^d$, we recall  $L_\Pc(n)=|n\Pc \cap  \ZZ^d|$ the Ehrhart polynomial of $\Pc$.
		Let $g_r(\Pc)$ be the $r$th coefficients of $L_\Pc(n)$ for $r=0,\ldots,d$.
		Martin Henk and Makoto Tagami gave lower bounds on the coefficients $g_r(\Pc)$ in terms of the volume of $\Pc$.
    	They proved that these bounds are best possible for $r \in \{1,2,d-2\}$.
    	We show that these bounds are also optimal for $r=3$ and $d-r$ even and we give a new best possible bound for $r=d-3$.
	\end{abstract}
	\maketitle
	\section*{INTRODUCTION}
	Let $\Pc \subset \RR^d$ be an integral convex polytope, that is, a convex polytope whose vertices have integer coordinates, of dimension $d$.
	Given integers $n=1,2,\ldots,$ we write $L_\Pc(n)$ 
	for the number of integer points belonging to $n\Pc$, 
	where $n\Pc=\left\{n\alpha : \alpha \in \Pc \right\}$.
	In other words,
	$$L_\Pc(n)=|n\Pc \cap  \ZZ^d|, \ \ \ n=1,2,\ldots .$$
	In the late 1950s, Ehrhart succeeded in proving that $L_\Pc(n)$ 
	is a polynomial in $n$ of degree $d$.
	We call $L_\Pc(n)$ the \textit{Ehrhart \ polynomial} of $\Pc$.
	We refer the readers to \cite{Beck,Ehrhart,Hibi}  for an introduction to the theory of Ehrhart polynomials.
    For $r=0,\ldots,d$, let $g_r(\Pc)$ be the $r$th coefficient of $L_\Pc(n)$.
	The following properties are known:
	\begin{itemize}	
		\item[(i)] $g_0=1$;
		\item[(ii)] $g_d(\Pc)=\text{vol}(\Pc)$;
		\item[(iii)] $g_{d-1}(\Pc)=\cfrac{1}{2}\sum\limits_{\Fc \ \text{facet of} \ \Pc}\cfrac{\text{vol}_{d-1}(\Fc)}{\text{det}(\text{aff}\Fc \cap \ZZ^d)}$ (\cite[Theorem 5.6]{Beck});
	\end{itemize}
	where $\text{vol}(\cdot)$ denotes the usual volume and $\text{vol}_{d-1}(\cdot)$ denotes ($d-1$)-dimensional volume,
	and det(aff$\Fc \cap \ZZ^d$) denotes the determinant of the ($d-1$)-dimensional sublattice contained in the affine hull of  $\Fc$.
    All other coefficients $g_r(\Pc)$, $1\leq r \leq d-2 $, have no such known explicit geometric meaning, except for special classes of  polytopes.  
 
 In \cite[Theorem 6]{upper} Ulrich Betke and Peter McMullen proved the following upper bounds on the coefficients $g_r(\Pc)$ in terms of the volume:
 $$g_r(\Pc) \leq (-1)^{d-r}\text{stirl}(d,r)\text{vol}(\Pc)+(-1)^{d-r-1}\cfrac{\text{stirl}(d,r+1)}{(d-1)!}, \  1\leq r \leq d-1,$$
 where stirl($d,i$) denote the Striling numbers of the first kind which can be defined via the identity $\prod_{i=0}^{d-1} (z-i)=\sum_{i=1}^d \text{stirl}(d,i)\,z^i$.
 
 For an integer $i$ and a variable $z$ we consider  the polynomial of degree $d$.
$$ c_i(z)=(z+i)(z+i-1)\cdots(z+i-(d-1))=d!\,\binom{z+i}{d},    $$
 and we denote its $r$th coefficient by $C^d_{r,i}$, $0\leq r\leq d$.
 For instance, $C_{d,i}^d=1$, and for $0\leq i\leq d-1$ we have
 $C_{0,i}^d=0$. For $d\geq 3$ we set  
$$ M_{r,d}=\min\{C_{r,i}^d : 1\leq i\leq d-2\}. $$
In \cite[Theorem 1.1]{lower} Martin Henk and Makoto Tagami proved the following lower bounds on the coefficients $g_r(\Pc)$ in terms of the volume:
\begin{align}
g_r(\Pc) \geq \frac{1}{d!}\left((-1)^{d-r}\text{stirl}(d+1,r+1)+(d!\text{vol}(\Pc)-1)M_{r,d}\right), \  1\leq r \leq d-1.
\end{align}
 	 In general, these bounds  are not best possible.
 	 However, it is known that in the cases $r \in \{1,2,d-2\}$, they are best possible for any volume.
 	
 	 In Section 1, we show that in the case $r=3$ and in the cases $d-r$ is even, these bounds are best possible. 
	 In Section 2, we give the following lower bound on $g_{d-3}(\Pc)$,
 	$$ g_{d-3}(\Pc) \geq \frac{1}{d!}\left(-\text{stirl}(d+1,d-2)+(d!\text{vol}(\Pc)-1)N_{d-3,d}\right),$$
 	where $N_{d-3,d}=\min\left\{C^d_{d-3,i}:\lceil (d-1)/2 \rceil \leq i \leq d-2 \right\}.$
 	 In particular, we show the bound is best possible and $N_{d-3,d} > M_{d-3,d}$, i.e., Henk and Tagami's bound on $g_{d-3}(\Pc)$ is not best possible.

	\section{The cases $r=3$ or $d-r$ is even }
Let $\Pc \subset \RR^d$ be an integral convex polytope of dimension $d$, and let $\partial \Pc$ denote the boundary of $\Pc$.
The generating function of the integral point enumerator, i.e., the formal power series
$$\text{Ehr}_\Pc(t)=1+\sum\limits_{n=1}^{\infty}L_\Pc(n)t^n$$
is called the Ehrhart series of $\Pc$.
Since $L_\Pc(n)$ is a polynomial of degree $d$, by a routine argument, it follows that the Ehrhart series of $\Pc$ can be expressed as a rational function of the form
$$\text{Ehr}_\Pc(t)=\frac{\delta_0+\delta_1t+\cdots + \delta_dt^d}{(1-t)^{d+1}}.$$
The sequence of the coefficients of the polynomial in the numerator 
$$\delta(\Pc)=(\delta_0,\delta_1,\ldots,\delta_d)$$
is called the \textit{$\delta$-vector} of $\Pc$.

The $\delta$-vector has the following properties:
\begin{itemize}
	\item[(i)] $\delta_0=1,\delta_1=|\Pc \cap \ZZ^d|-(d+1)$ and $\delta_d=|(\Pc\setminus \partial \Pc) \cap \ZZ^d|$. Hence, $\delta_1 \geq \delta_d$; 
	\item[(ii)] Each $\delta_i$ is nonnegative (\cite{RS_DRCP});
	\item[(iii)] If $\delta_d \neq 0$, then one has $\delta_1 \leq \delta_i$ for every $1 \leq i \leq d-1$ (\cite{H_LBTEP});
	\item[(iv)] The leading coefficient $(\sum_{i=0}^{d}\delta_i)/d!$ of $L_\Pc(n)$ is equal to the usual volume of $\Pc$, i.e., $d!\text{vol}(\Pc)=\sum_{i=0}^{d}\delta_i$ (\cite[Corollary 3.20, 3.21]{Beck}):
\end{itemize}

There are two well-known inequalities on $\delta$-vectors.
Let $s = \max\left\{i : \delta_i \neq 0 \right\}$.
One inequality is
\begin{align}
\delta_0+\delta_1+\cdots+\delta_{i}\leq \delta_s+\delta_{s-1}+\cdots+\delta_{s-i}, \ \ 0 \leq i \leq \lfloor s/2 \rfloor,
\end{align}
which was proved by Stanley \cite{RS_OHF}, and another one is
\begin{align}
\delta_d+\delta_{d-1}+\cdots+\delta_{d-i}\leq \delta_1+\delta_2+\cdots+\delta_{i+1}, \ \ 0 \leq i \leq\lfloor(d-1)/2\rfloor ,
\end{align}
which appears in the work of Hibi \cite[Remark 1.4]{H_LBTEP}.
Also, there are more recent and more general results on $\delta$-vectors by Alan Stapledon in \cite{Stap}.

We can express the coefficients $g_r(\Pc)$ of the Ehrhart polynomial $L_\Pc(n)$ by using the $\delta$-vector $\delta(\Pc)$ (\cite[Proof of Theorem 1.1]{lower}).
In fact,
 $$g_r(\Pc)=\frac{1}{d!}\sum\limits_{i=0}^{d}\delta_iC_{r,d-i}^d.$$
 
We repeatedly use the following lemmas in this paper.
\begin{lemma}[{\cite[Theorem 1.1]{Higashitani}}]
	\label{higashitani}
	Let $m,d,k \in \ZZ_{> 0}$ be arbitrary positive integers satisfying
	$m \geq 1, \; d \geq 2 \; \text{ and } \; 
	1 \leq k \leq \lfloor (d+1)/2 \rfloor.$
	Then there exists an integral convex polytope $\Pc$ of dimension $d$
	such that $\delta_0=1,\delta_k=m$ and for each $i \notin \{0,k\}$, $\delta_i=0$, where we let $\delta(\Pc)=(\delta_0,\ldots,\delta_d)$ be the $\delta$-vector of $\Pc.$
\end{lemma}
\begin{lemma}[{\cite[Lemma 2.2]{lower}}]
	\label{sym}
	$C_{r,i}^d=(-1)^{d-r}C_{r,d-1-i}^d$ for $0\leq i\leq d-1$.
\end{lemma}

First, we show that in the case $r=3$ the bound (0.1) is best possible for any volume.
In fact,
\begin{theorem}
	\label{3}
	Let $d$ be an integer with $d \geq 6$ and let $\Pc$ be an integral convex polytope of dimension $d$.
	Then
	\begin{displaymath}
	g _3(\Pc) \geq \left\{        
	\begin{aligned} 
	&\frac{1}{d!}\left((-1)^{d-1}\textnormal{stirl}(d+1,4)+(d!\textnormal{vol}(\Pc)-1)C^d_{3,d-3}\right) , & 6 \leq d \leq 9, \\
	&\frac{1}{d!}\left((-1)^{d-1}\textnormal{stirl}(d+1,4)+(d!\textnormal{vol}(\Pc)-1)C^d_{3,d-2}\right), & d \geq 10, \\
	\end{aligned}
	\right.
	\end{displaymath}
	and the bound is best possible for any volume.	
\end{theorem}

In order to prove Theorem \ref{3}, we use the following lemma.
	\begin{lemma}
		\label{lem3}
		\begin{displaymath}
		M_{3,d}= \left\{
			\begin{aligned}
			& C^d_{3,d-3}, & 6\leq d \leq 9, \\
			& C^d_{3,d-2}, & d \geq 10.\\
			\end{aligned}
			\right.
		\end{displaymath}
	\end{lemma}
	\begin{proof}		
		For $1 \leq k \leq \lfloor (d+1)/2\rfloor -1$,
		we set 
		\begin{displaymath}
		\begin{aligned}
		f_k(z)&=(z+k)(z+k-1)\cdots(z+1)z(z-1)\cdots(z-k+1)(z-k),\\
		g_k(z)&=(z-k-1)(z-k-2)\cdots(z+k-d+2)(z+k-d+1).\\
		\end{aligned} 
		\end{displaymath}
		Then $c_k(z)=f_k(z)g_k(z)$.
		Also, for  $1 \leq k \leq \lfloor(d+1)/2 \rfloor -2$,
		we set 
		$$I_k=\left\{i \in \ZZ \ :\  k+1\leq i \leq d-k-1 \right\}.$$
		
		We assume that $d \geq 19$.
	If $1 \leq k \leq \lfloor (d+1)/2 \rfloor -2$, then
	\begin{displaymath}
	\begin{aligned}
	f_k(z)&=(-1)^{k}(k!)^2z+(-1)^{k-1}(k!)^2\sum\limits_{1 \leq i \leq k}
	\cfrac{1}{i^2}z^3+\text{upper terms},\\
	g_k(z)=&(-1)^{d-2k-1}\cfrac{(d-k-1)!}{k!}+(-1)^{d-2k}\frac{(d-k-1)!}{k!}\sum\limits_{i\in I_k}\cfrac{1}{i}z\\
	&+(-1)^{d-2k-1}\cfrac{(d-k-1)!}{k!}\sum\limits_{\left\{i,j\right\}\subset I_k}\cfrac{1}{ij}z^2+\text{upper terms}.
	\end{aligned} 
	\end{displaymath}
Hence for $1 \leq k \leq\lfloor (d+1)/2\rfloor -1$,
$$C^d_{3,k}=(-1)^{d-k-1}k!(d-k-1)!\left(\sum\limits_{\left\{i,j\right\}\subset I_k}\cfrac{1}{ij}-\sum\limits_{1 \leq i \leq k}
\cfrac{1}{i^2}\right).$$

We show that for $2 \leq k \leq  \lfloor(d+1)/2\rfloor -1$, $|C_{3,1}^d|>|C_{3,k}^d|$.
Since $\sum\limits_{\left\{i,j\right\}\subset I_1}\frac{1}{ij}-1>0$,
we have $|C_{3,1}^d|=(d-2)!(\sum\limits_{\left\{i,j\right\}\subset I_1}\frac{1}{ij}-1).$
Hence since $(d-2)!>k!(d-k-1)!$ and $d \geq 19$, we have
\begin{displaymath}
\begin{aligned}
& (d-2)!\left(\sum\limits_{\left\{i,j\right\}\subset I_1}\cfrac{1}{ij}-1\right)+k!(d-k-1)!\left(\sum\limits_{\left\{i,j\right\}\subset I_k}\cfrac{1}{ij}-\sum\limits_{1 \leq i \leq k}\cfrac{1}{i^2}\right)\\
>&(d-2)!\left(\sum\limits_{\left\{i,j\right\}\subset I_1}\cfrac{1}{ij}-1-\sum\limits_{1 \leq i \leq \infty}\cfrac{1}{i^2}\right)\\
=&(d-2)!\left(\sum\limits_{\left\{i,j\right\}\subset I_1}\cfrac{1}{ij}-1-\cfrac{\pi^2}{6}\right)\\
>& 0,
\end{aligned}
\end{displaymath}
and
\begin{displaymath}
\begin{aligned}
& (d-2)!\left(\sum\limits_{\left\{i,j\right\}\subset I_1}\cfrac{1}{ij}-1\right)-k!(d-k-1)!\left(\sum\limits_{\left\{i,j\right\}\subset I_k}\cfrac{1}{ij}-\sum\limits_{1 \leq i \leq k}\cfrac{1}{i^2}\right)\\
>&(d-2)!\left(\sum\limits_{\left\{i,j\right\}\subset I_1}\cfrac{1}{ij}-1-\sum\limits_{\left\{i,j\right\}\subset I_k}\cfrac{1}{ij}\right)\\
>& 0.
\end{aligned}
\end{displaymath}
Therefore, since $|C_{3,1}^d|+C_{3,k}^d>0$ and $|C_{3,1}^d|-C_{3,k}^d>0$, we have $|C_{3,1}^d|>|C_{3,k}^d|$.
Thus, since $C_{3,1}^d>0$, by Lemma \ref{sym}, we have $M_{3,d}=C_{3,d-2}^d$.

Finally, for $6 \leq d \leq 18$ and for $1 \leq k \leq d-2$, by computing $C_{3,k}^d$, we have
\begin{displaymath}
M_{3,d}= \left\{
\begin{aligned}
& C^d_{3,d-3}, & 6\leq d \leq 9, \\
& C^d_{3,d-2}, & 10 \leq d \leq 18,\\
\end{aligned}
\right.
\end{displaymath}
as desired.
\end{proof}
Now, we prove Theorem \ref{3}.
\begin{proof}[Proof of Theorem \ref{3}]
We set 
\begin{displaymath}
	j_d=\left\{
	\begin{aligned}
	&2, & 6 \leq d \leq 9,\\
	&1, &d \geq 10.
	\end{aligned}
\right.
\end{displaymath}
Then by Lemma \ref{lem3}, we have $M_{3,d}=C^d_{3,d-1-j_d}$.
Since $j_d+1 \leq \lfloor(d+1)/2\rfloor$, by Lemma \ref{higashitani}, for any volume, there exists an integral convex polytope $\Pc$ such that $\delta_0=1$ and $\delta_{j_d+1}=d!\text{vol}(\Pc)-1$ and for each $i \notin \left\{0,j_d+1\right\}$, $\delta_i=0$,
where we let $\delta(\Pc)=(\delta_0,\ldots,\delta_d)$ be the $\delta$-vector of $\Pc.$
Hence since $C_{3,d}^d=(-1)^{d-1}\text{stirl}(d+1,4)$, we have
\begin{displaymath}
\begin{aligned}
 d!g_3(\Pc)&=\sum\limits_{i=0}^{d}\delta_iC_{3,d-i}^d\\
 &=(-1)^{d-1}\text{stirl}(d+1,4)+(d!\text{vol}(\Pc)-1)C^d_{3,d-1-j_d}\\
 &=(-1)^{d-1}\text{stirl}(d+1,4)+(d!\text{vol}(\Pc)-1)M_{3,d},
\end{aligned}
\end{displaymath}
as desired.
\end{proof}

Next, we show that in the cases that $d-r$ is even, the bounds (0.1) are best possible for any volume.
In fact,
\begin{theorem}
	\label{even}
	Let  $d$ and $r$ be integers with $d \geq 6$ and $3 \leq r \leq d-3$, and let $\Pc \subset \RR^d$ be an integral convex polytope of dimension $d$.
	Suppose that $d-r$ is even.
	Then
	$$g_r(\Pc) \geq \cfrac{1}{d!}\left(\textnormal{stirl}(d+1,r+1)+(d!\textnormal{vol}(\Pc)-1)M_{r,d}\right),$$
	and the bound is best possible for any volume.
\end{theorem}
\begin{proof}
	We let $i$ be an integer with $1 \leq i \leq d-2$ such that  $C_{r,i}^d=M_{r,d}$.
	Since $d-r$ is even, by Lemma \ref{sym}, we have $C_{r,i}^d=C_{r,d-i-1}^d$.
	We set $j:= \max\left\{i,d-i-1\right\}$.
	Then $1 \leq d-j \leq \lfloor(d+1)/2 \rfloor$.
	Hence by Lemma \ref{higashitani}, for any volume, there exists an integral convex polytope $\Pc$ such that $\delta_0=1$ and $\delta_{d-j}=d!\text{vol}(\Pc)-1$ and for each $i \notin \left\{0,d-j\right\}$, $\delta_i=0$.
	Therefore, since $C_{r,d}^d=\text{stirl}(d+1,r+1)$, we have
	\begin{displaymath}
	\begin{aligned}
	d!g_r(\Pc) 
	&=\textnormal{stirl}(d+1,r+1)+(d!\textnormal{vol}(\Pc)-1)C^d_{r,j}\\
	&=\textnormal{stirl}(d+1,r+1)+(d!\textnormal{vol}(\Pc)-1)M_{r,d},
	\end{aligned}
	\end{displaymath}
	as desired.
\end{proof}
\section{A new lower bound on $g_{d-3}(\Pc)$}
We assume that $d\geq 7$ and  $r=d-3$.
Then since $d-r$ is odd and $r \geq 4$, it is not known whether the bound (0.1) on $g_{d-3}(\Pc)$ is best possible for any volume.
In this section, we give a new lower bound on $g_{d-3}(\Pc)$.
In particular, we show the bound is best possible, i.e., the bound (0.1) on  $g_{d-3}(\Pc)$ is not best possible.

We set
$$N_{d-3,d}=\min\left\{C^d_{d-3,i}\ : \ \lceil (d-1)/2 \rceil \leq i \leq d-2 \right\}.$$ 
Then $N_{d-3,d} \geq M_{d-3,d}$.

In the following theorem, we give a new lower bound on $g_{d-3}(\Pc)$.

\begin{theorem}
	\label{d-3}
	Let $d$ be an integer with $d \geq 7$ and $\Pc$ an integral convex polytope of dimension $d$.
	Then
	$$g_{d-3}(\Pc) \geq \cfrac{1}{d!}\left(-\textnormal{stirl}(d+1,d-2)+(d!\textnormal{vol}(\Pc)-1)N_{d-3,d}\right),$$
	and the bound is best possible for any volume.
	In particular, $M_{d-3,d}<N_{d-3,d}$.
\end{theorem}
In order to prove Theorem \ref{d-3}, we use the following lemma.
\begin{lemma}
	\label{kak}
	For $0 \leq k \leq \lfloor(d+1)/2\rfloor-2$,
	if $C_{d-3,k}^d\leq 0$, then $C_{d-3,k+1}^d \geq C_{d-3,k}^d$,
	and if $C_{d-3,k}^d\geq 0$, then $C_{d-3,k+1}^d \geq 0$.
\end{lemma}

Before proving Lemma \ref{kak}, we show the following lemma.
\begin{lemma}
	\label{poly}
	Let $x,y \in \ZZ$ and $y>0$.
	Then
	$$\sum\limits_{-y\leq a<b \leq y}(x+a)(x+b)=\binom{2y+1}{2}x^2-\cfrac{1}{4}\binom{2y+2}{3},$$
	and
	$$\sum\limits_{-y\leq a<b<c \leq y}(x+a)(x+b)(x+c)=\binom{2y+1}{3}x^3-\binom{2y+2}{4}x.$$
	
\end{lemma}
\begin{proof}
	By using induction on $y$, it immediately follows.
\end{proof}
Now, we prove Lemma \ref{kak}.
\begin{proof}[Proof of Lemma \ref{kak}]
	We let $f_k(z)$, $g_k(z)$ and  $I_k$ be as in the proof of Lemma \ref{lem3}.
	 
	If $1 \leq k \leq \left\lfloor (d+1)/2 \right\rfloor -3$, then
	\begin{displaymath}
	\begin{aligned}
	f_k(z)=&z^{2k+1}-\sum\limits_{1 \leq i \leq k}i^2z^{2k-1}+\text{lower terms},\\
	g_k(z)=&z^{d-2k-1}-\sum\limits_{i\in I_k}iz^{d-2k-2}+\sum\limits_{\left\{i,j\right\}\subset I_k}ijz^{d-2k-3}-\sum\limits_{\left\{i,j,l\right\}\subset I_k}ijlz^{d-2k-4}+\text{lower terms}.
	\end{aligned} 
	\end{displaymath}
	Hence for $1 \leq k \leq \left\lfloor (d+1)/2 \right\rfloor -1$,
	$$C_{d-3,k}^d=-\sum\limits_{\left\{i,j,l\right\}\subset I_k}ijl+\left(\sum\limits_{1 \leq i \leq k}i^2\right)\cdot\left(\sum\limits_{i\in I_k}i\right).$$
  Since $C_{d-3,0}^d=-\sum_{\left\{i,j,l\right\}\subset I_0}ijl<0$
  and  $I_1 \subset I_0$,	
	we have $C_{d-3,1}^d-C_{d-3,0}^d>0$.
	Hence we should show the cases $1 \leq k \leq \lfloor(d+1)/2\rfloor-2$.
	
	First, we show the cases $d$ is even.
	If $k=\left\lfloor (d+1)/2 \right\rfloor -1$, then $C_{d-3,k}^d>0$.
	Hence we should show the cases $1 \leq k \leq \lfloor(d+1)/2\rfloor-3$.
	We set $x=d/2$ and for $1 \leq k \leq \lfloor(d+1)/2\rfloor-2$, 
	we set $J_k=\left\{i \in \ZZ \ : \  -x+k+1\leq i \leq x-k-1 \right\}.$
	For $1 \leq k \leq \lfloor(d+1)/2\rfloor-2$, by Lemma \ref{poly}, we have
		\begin{displaymath}
		\begin{aligned}
			C_{d-3,k}^d 
			=&-\sum\limits_{\left\{i,j,l\right\}\subset J_k}\left(x+i\right)\left(x+j\right)\left(x+l\right)+\left(\sum\limits_{1 \leq i \leq k}i^2\right)\cdot\left(\sum\limits_{i\in I_k}i\right)\\
			=&-\left(\binom{d-2k-1}{3}x^3-\binom{d-2k}{4}x\right)+\cfrac{k(k+1)(2k+1)}{6}\cfrac{d(d-2k-1)}{2}\\
			=&-\cfrac{(d-2)(d-1)d(d-2k-1)(4k^2-4dk+4k+d^2-3d)}{48}.
		\end{aligned} 
		\end{displaymath}
		If $-\frac{\sqrt{d+1}-d+1}{2} \leq k \leq \frac{\sqrt{d+1}+d-1}{2}$, then $4k^2-4dk+4k+d^2-3d \leq 0$ and hence $C_{d-3,k}^d\geq0$, and if $1 \leq k \leq -\frac{\sqrt{d+1}-d+1}{2} $, then $C_{d-3,k}^d\leq0$.
		Therefore, if $C_{d-3,k}^d\geq0$, we have $C_{d-3,k+1}^d\geq0$.
		If $1 \leq k \leq -\frac{\sqrt{d+1}-d+1}{2}$ and $1 \leq k \leq \lfloor(d+1)/2\rfloor-3$, then $C_{d-3,k}^d\leq0$ and
		\begin{displaymath}
		\begin{aligned}
		C_{d-3,k+1}^d-C_{d-3,k}^d=\cfrac{(d-2)(d-1)d(12k^2-12dk+24k+3d^2-13d+12)}{24}.
		\end{aligned} 
		\end{displaymath}
		Since $1 \leq k \leq -\frac{\sqrt{d+1}-d+1}{2} \leq \frac{3d-6-\sqrt{3d}}{6}$, we have $C_{d-3,k+1}^d-C_{d-3,k}^d \geq 0$.
		
	Next, we show the cases $d$ is odd.
	If $k= (d+1)/2 -1$ or $k=(d+1)/2 -2$, then $C_{d-3,k}^d\geq0$.
	Hence we should show the cases $1 \leq k \leq (d+1)/2-4$.
	We set $y=(d+1)/2$ and for $1 \leq k \leq (d+1)/2-3$,
	we set $L_k=\left\{i \in \ZZ \ : \  -y+k+2\leq i \leq y-k-2 \right\}.$
	For $1 \leq k \leq (d+1)/2-3$, by Lemma \ref{poly}, we have
	\begin{displaymath}
	\begin{aligned}
C_{d-3,k}^d
	=&-\sum\limits_{\left\{i,j,l\right\}\subset L_k \cup \left\{-y+k+1\right\}}\left(y+i\right)\left(y+j\right)\left(y+l\right)+\left(\sum\limits_{1 \leq i \leq k}i^2\right)\cdot\left(\sum\limits_{i\in I_k}i\right)\\
    =&-(k+1)\sum\limits_{\left\{i,j\right\}\subset L_k}\left(y+i\right)\left(y+j\right)-\sum\limits_{\left\{i,j,l\right\}\subset L_k}\left(y+i\right)\left(y+j\right)\left(y+l\right)\\
    &+\left(\sum\limits_{1 \leq i \leq k}i^2\right)\cdot\left(\sum\limits_{i\in I_k}i\right)\\
	=&-\cfrac{(d-2)(d-1)d(d-2k-1)(4k^2-4dk+4k+d^2-3d)}{48}.
	\end{aligned} 
	\end{displaymath}
    Thus it follows by the same argument, as desired.
\end{proof}
Finally, we prove Theorem \ref{d-3}.
\begin{proof}[Proof of Theorem \ref{d-3}]
	Let $\delta(\Pc)=(\delta_0,\ldots,\delta_d)$ be the $\delta$-vector of $\Pc$.
	We assume that $d$ is even.
	Then by Lemma \ref{sym}, we have
\begin{displaymath}
\begin{aligned}
d!g_{d-3}(\Pc) 
					  =&C_{d-3,d}^d+\sum\limits_{i=1}^{\frac{d}{2}}(\delta_i-\delta_{d-i+1})C_{d-3,d-i}^d\\
					  =&C_{d-3,d}^d+\sum\limits_{i=1}^{\frac{d}{2}-1}\left(\sum\limits_{j=1}^{i}(\delta_j-\delta_{d-j+1})(C_{d-3,d-i}^d-C_{d-3,d-i-1}^d)\right)\\
					  &+\sum\limits_{j=1}^{\frac{d}{2}}(\delta_j-\delta_{d-j+1})C_{d-3,\frac{d}{2}}^d.\\
\end{aligned} 
\end{displaymath}
By Lemmas \ref{sym} and \ref{kak}, we have $N_{d-3,d}<0$ and there exists an integer $t$ such that for $d/2 \leq i \leq t$, $C_{d-3,i}^d \leq 0$, and for $t+1 \leq i \leq d-1$, $C_{d-3,i}^d-C_{d-3,i-1}^d \geq 0$.
Hence by the inequality (1.2), we have
\begin{displaymath}
\begin{aligned}
d!g_{d-3}(\Pc)
\geq& C_{d-3,d}^d+\sum\limits_{i=d-t}^{\frac{d}{2}-1}\left(\sum\limits_{j=1}^{i}(\delta_j-\delta_{d-j+1})(C_{d-3,d-i}^d-C_{d-3,d-i-1}^d)\right)\\
&+\sum\limits_{j=1}^{\frac{d}{2}}(\delta_j-\delta_{d-j+1})C_{d-3,\frac{d}{2}}^d\\
=&C_{d-3,d}^d+\sum\limits_{j=1}^{d-t}(\delta_j-\delta_{d-j+1})C_{d-3,t}^d+\sum\limits_{i=d-t+1}^{\frac{d}{2}}(\delta_i-\delta_{d-i+1})C_{d-3,d-i}^d.\\
\end{aligned} 
\end{displaymath}
If for $d-t+1 \leq i \leq d/2$, $\delta_i-\delta_{d-i+1} \leq 0$, then since $C_{d-3,d-i}^d \leq 0$, we have $(\delta_i-\delta_{d-i+1})C_{d-3,d-i}^d \geq 0$.
Also, if for $d-t+1 \leq i \leq d/2$, $\delta_i-\delta_{d-i+1} \geq 0$, then $(\delta_i-\delta_{d-i+1})C_{d-3,d-i}^d \geq (\delta_i-\delta_{d-i+1})N_{d-3,d}$.
Hence since $C_{d-3,d}^d=-\text{stirl}(d+1,d-2)$ and $N_{d-3,d}<0$, by the inequality (1.2), we have
\begin{displaymath}
\begin{aligned}
d!g_{d-3}(\Pc) & \geq C_{d-3,d}^d+\sum\limits_{j=1}^{d-t}(\delta_j-\delta_{d-j+1})N_{d-3,d}+\sum\limits_{i \in I'}(\delta_i-\delta_{d-i+1})N_{d-3,d}\\
& \geq -\text{stirl}(d+1,d-2)+\sum\limits_{i=1}^{d}\delta_iN_{d-3,d}\\
& = -\text{stirl}(d+1,d-2)+(d!\text{vol}(\Pc)-1)N_{d-3,d},
\end{aligned} 
\end{displaymath}
where $I'=\{i \in \ZZ \  : \  d-t+1 \leq i \leq d/2, \delta_i-\delta_{d-i+1} \geq 0 \}$.

Next, we assume that $d$ is odd.
Then by Lemma \ref{sym}, we have
\begin{displaymath}
\begin{aligned}
d!g_{d-3}(\Pc)=&\sum\limits_{i=0}^{d}\delta_iC_{d-3,d-i}^d\\
=&C_{d-3,d}^d+\sum\limits_{i=1}^{\frac{d+1}{2}-2}\left(\sum\limits_{j=1}^{i}(\delta_j-\delta_{d-j+1})(C_{d-3,d-i}^d-C_{d-3,d-i-1}^d)\right)\\
&+\sum\limits_{j=1}^{\frac{d+1}{2}-1}(\delta_j-\delta_{d-j+1})C_{d-3,\frac{d+1}{2}}^d+\delta_{\frac{d+1}{2}}C_{d-3,\frac{d-1}{2}}^d.\\
\end{aligned} 
\end{displaymath}
Hence it follows by the same argument.

By Lemma \ref{higashitani}, it follows that this bound is best possible.

Finally, we show that $M_{d-3,d}<N_{d-3,d}$.
In particular, we show that $C_{d-3,1}^d<N_{d-3,d}$.
For $2 \leq k \leq \lfloor (d+1)/2 \rfloor-3$,
\begin{displaymath}
\begin{aligned}
C_{d-3,1}^d+C_{d-3,k}^d&=-\cfrac{(d-2)(d-1)d(d-k-2)(4k^2-2dk-2k+d^2-5d+6)}{24}<0.
\end{aligned}
\end{displaymath}
Also, if $d$ is odd and $k=(d+1)/2-2$, then $C_{d-3,1}^d+C_{d-3,k}^d<0$,
and if $d$ is odd and $k=(d+1)/2-1$, then 
 \begin{displaymath}
 \begin{aligned}
 C_{d-3,1}^d+C_{d-3,k}^d&=-\cfrac{(d-3)(d-2)(d-1)d(d^2-7d+8)}{48}<0,
 \end{aligned}
 \end{displaymath}
and if $d$ is odd even $k=d/2-2$, then
 \begin{displaymath}
 \begin{aligned}
 C_{d-3,1}^d+C_{d-3,k}^d&=-\cfrac{(d-2)d^2(d-1)(d^2-10d+26)}{48}<0.
 \end{aligned}
 \end{displaymath}
Hence by Lemmas \ref{sym} and \ref{kak}, we have $C_{d-3,1}^d<N_{d-3,d}$, as desired.
\end{proof}

For $d\geq 3$ we set  
$$N_{r,d}=\min\left\{C^d_{r,i}\ : \ \lceil (d-1)/2 \rceil \leq i \leq d-2 \right\}.$$ 
We recall the following lemma.

\begin{lemma}[{\cite[Proposition 2.1]{lower}}]
	Let $d \geq 3$.
	Then\\
	\textnormal{(i)} $M_{1,d}=C_{1,d-2}^d$,\\
	\textnormal{(ii)} $M_{2,d}=C_{2,d-2}^d$,\\
	\textnormal{(iii)} $M_{d-2,d}=C_{d-2,\lceil\frac{d-1}{2}\rceil}^d$.
\end{lemma}

If $r \in \left\{1,2,3,d-3,d-2\right\}$ or $d-r$ is even, then by this lemma and by the results so far, we have $M_{r,d}=N_{r,d}$.
Hence we immediately have the following corollary.
\begin{corollary}
	Let $\Pc \subset \RR^d, d \geq 3$. 
	Assume that $r \in \left\{1,2,3,d-3,d-2\right\}$ or $d-r$ is even.
	Then we have
	$$g_{r}(\Pc) \geq \cfrac{1}{d!}\left((-1)^{d-r}\textnormal{stirl}(d+1,r+1)+(d!\textnormal{vol}(\Pc)-1)N_{r,d}\right),$$
	and the bound is best possible for any volume.
\end{corollary}

Then the following conjecture naturally occurs:

\begin{conjecture}
	\label{conj}
	Let $\Pc \subset \RR^d, d \geq 3 $. 
	Then for $r=1,\ldots,d-2$ 
	$$g_{r}(\Pc) \geq \cfrac{1}{d!}\left((-1)^{d-r}\textnormal{stirl}(d+1,r+1)+(d!\textnormal{vol}(\Pc)-1)N_{r,d}\right),$$
	and the bounds are best possible for any volume.	
\end{conjecture}

We set the sequence $C_{r,d}=(C_{r,\lceil\frac{d-1}{2}\rceil}^d,\ldots,C_{r,d-1}^d)$.
As we see in the proof of Theorem \ref{d-3}, in order to solve Conjecture \ref{conj}, it seems efficient to study the properties of the sequence $C_{r,d}$.
The following proposition may be important to solve this conjecture.

\begin{proposition}
	\label{prop}
Let $\Pc \subset \RR^d, d \geq 3$. 
Assume that $d-r$ is odd and the sequence $C_{r,d}$ satisfies the following condition\textnormal{:}
There exists an integer $t$ with $\lceil(d-1)/2 \rceil+1 \leq t \leq d-1$ such that 
$C_{r,i}^d \geq C_{r,i-1}^d$ for $t \leq i \leq d-1$ and $N_{r,d}=\min\left\{C_{r,i}^d\ : \ \lceil(d-1)/2\rceil \leq i \leq t-1\right\}=\min\left\{-|C_{r,i}^d|:\lceil(d-1)/2\rceil \leq i \leq t-1\right\}$.
Then we have
$$g_{r}(\Pc) \geq \cfrac{1}{d!}\left((-1)^{d-r}\textnormal{stirl}(d+1,r+1)+(d!\textnormal{vol}(\Pc)-1)N_{r,d}\right).$$
And the bound is best possible for any volume.
\end{proposition}
\begin{proof}
If $\delta_{d-i}-\delta_{i+1} \leq 0$ for $\lceil(d-1)/2\rceil \leq i \leq t-1$,
then since $N_{r,d} \leq -C_{r,i}^d$, we have $(\delta_{d-i}-\delta_{i+1})C_{r,i}^d \geq (\delta_{d-i}-\delta_{i+1})N_{r,d}$.
Also, if $\delta_{d-i}-\delta_{i+1} \geq 0$  for $\lceil(d-1)/2\rceil \leq i \leq t-1$, 
then since $N_{r,d} \leq C_{r,i}^d$, we have $(\delta_{d-i}-\delta_{i+1})C_{r,i}^d \geq (\delta_{d-i}-\delta_{i+1})N_{r,d}$.
Hence since $N_{r,d}\leq 0$, by the same argument of the proof of Theorem \ref{d-3}, we have 
$$g_{r}(\Pc) \geq \cfrac{1}{d!}\left((-1)^{d-r}\textnormal{stirl}(d+1,r+1)+(d!\textnormal{
	vol}(\Pc)-1)N_{r,d}\right).$$
By Lemma \ref{higashitani}, the bound is best possible for any volume.
\end{proof}

In lower dimensions, by computing $C_{r,d}$ we can know whether $C_{r,d}$ satisfies the condition of Proposition \ref{prop}.
In fact, we know from computational experiment that for $d\leq 1000$, $C_{r,d}$ satisfies the condition.
Hence by results of this paper, in particular, by Theorem \ref{even}, the statement of Conjecture \ref{conj} holds for $d \leq 1000$.

\begin{acknowledgement}
	The author would like to thank referees for their helpful comments.
	He would also like to thank Professor Takayuki Hibi and Professor Akihiro Higashitani for helping him in writing this paper.
\end{acknowledgement}


\begin{thebibliography}{30}
		\bibitem{Beck}
		M. Beck and S. Robins,
		Computing the continuous discretely,
		Undergraduate Texts in Mathematics, Springer, 2007.
		
		\bibitem{upper}
		U. Betke and P. McMullen, Lattice points in lattice polytopes,
		\textit{Monatsh. Math.} \textbf{99}(1985), 253--265. 
		
		
		\bibitem{Ehrhart}
		E. Ehrhart,
		Polyn\^{o}mes Arithm\'{e}tiques et M\'{e}thode des Poly\`{e}dres en Combinatoire,
		Birkh\"{a}user, Boston/Basel/Stuttgart, 1977.
		
		\bibitem{lower}
		M. Henk and M. Tagami,
		Lower bounds on the coefficients of Ehrhart polynomials,
		\textit{Europ. J. Combinatorics} \textbf{30}(2009), 70--83.   
		
		\bibitem{Hibi}
		T. Hibi,
		Algebraic Combinatorics on Convex Polytopes,
		Carslaw Publications, Glebe, N.S.W., Australia, 1992.
		
		\bibitem{H_LBTEP}
		T.  Hibi,
		A lower bound theorem for Ehrhart polynomials of convex polytopes,
		\textit{Adv. in Math.} \textbf{105} (1994), 162--165. 
		
		\bibitem{Higashitani}
		A. Higashitani,
		Counterexamples of the conjecture on roots of Ehrhart polynomials,
		\textit{Discrete Comput. Geom. } \textbf{47} (2012), 618--623. 
		
		
		\bibitem{RS_DRCP} 
		R. P. Stanley,
		Decompositions of rational convex polytopes,
		\textit{Annals of Discrete Math. } \textbf{6} (1980), 333--342. 
		
		
		\bibitem{RS_OHF}
		R. P. Stanley,
		On the Hilbert function of a graded Cohen-Macaulay domain,
		\textit{J. Pure and Appl.  Algebra} \textbf{73} (1991), 307--314. 
		
		
	
		\bibitem{Stap}
		A. Stapledon,
		Inequalities and Ehrhart $\delta$-vectors,
		\textit{Trans. Amer. Math. Soc.} \textbf{361}(2009), 5615--5626.   
		
		
	\end{thebibliography}
\end{document}